\documentclass[11pt]{amsart}
\usepackage[utf8]{inputenc}

\usepackage{amsmath}
\usepackage{amssymb}
\usepackage{amsthm}
\usepackage{bbm}
\newcommand{\intdiv}{\hspace{2pt}}
\newcommand{\setdiv}{\hspace{2pt}\rvert\hspace{2pt}}

\newcommand{\biggsetdiv}{\hspace{2pt}\bigg\rvert\hspace{2pt}}

\newcommand{\abs}[1]{\left\lvert#1\right\rvert}
\newcommand{\norm}[2][]{{\left\| #2 \right\|}_{#1}}

\newcommand{\dims}{n} %abstracts the dimension of the space, so that this can be changed at a singular spot based on preference
\newcommand{\stdindex}{k} %abstracts the standard index.  This enables changing dims without collisions

\newcommand{\bbn}[1][]{{\mathbb{N}^{#1}}}
\newcommand{\bbr}{{\mathbb{R}}}
\newcommand{\bbrn}[1][\dims]{{\mathbb{R}^{#1}}}

\newcommand{\gln}[1][\dims]{\text{GL}_#1}

\newcommand{\calsn}[1][\dims-1]{{\mathcal{S}^{#1}}}
\newcommand{\calm}{{\mathcal{M}}}

\newcommand{\calh}{{\mathcal{H}}}
\newcommand{\calhn}[1][\dims]{{\calh^#1}}
\newcommand{\calhnd}[1][\stdindex]{{\calh_#1^\dims}}
\newcommand{\calp}[1][]{{\mathcal{P}_{#1}}}
\newcommand{\calr}{{\mathcal{R}}}
\newcommand{\cali}{{\mathcal{I}}}
\newcommand{\calip}{{\mathcal{I}^*}}
\newcommand{\calc}{{\mathcal{C}}}

\newcommand{\vol}{\text{vol}}

\newenvironment{problem}[1]{\vspace{5pt}\\ \noindent\textbf{Problem #1.} }{\\}
\newenvironment{theorem}[1]{\vspace{5pt} \noindent\textbf{Theorem #1.}}{}
\newenvironment{lemma}[1]{\vspace{5pt} \noindent\textbf{Lemma #1.} }{\\}

\makeatletter
\newcommand*{\mint}[1]{%
  \mint@l{#1}{}%
}
\newcommand*{\mint@l}[2]{%
  \@ifnextchar\limits{%
    \mint@l{#1}%
  }{%
    \@ifnextchar\nolimits{%
      \mint@l{#1}%
    }{%
      \@ifnextchar\displaylimits{%
        \mint@l{#1}%
      }{%
        \mint@s{#2}{#1}%
      }%
    }%
  }%
}
\newcommand*{\mint@s}[2]{%
  \@ifnextchar_{%
    \mint@sub{#1}{#2}%
  }{%
    \@ifnextchar^{%
      \mint@sup{#1}{#2}%
    }{%
      \mint@{#1}{#2}{}{}%
    }%
  }%
}
\def\mint@sub#1#2_#3{%
  \@ifnextchar^{%
    \mint@sub@sup{#1}{#2}{#3}%
  }{%
    \mint@{#1}{#2}{#3}{}%
  }%
}
\def\mint@sup#1#2^#3{%
  \@ifnextchar_{%
    \mint@sup@sub{#1}{#2}{#3}%
  }{%
    \mint@{#1}{#2}{}{#3}%
  }%
}
\def\mint@sub@sup#1#2#3^#4{%
  \mint@{#1}{#2}{#3}{#4}%
}
\def\mint@sup@sub#1#2#3_#4{%
  \mint@{#1}{#2}{#4}{#3}%
}
\newcommand*{\mint@}[4]{%
  \mathop{}%
  \mkern-\thinmuskip
  \mathchoice{%
    \mint@@{#1}{#2}{#3}{#4}%
        \displaystyle\textstyle\scriptstyle
  }{%
    \mint@@{#1}{#2}{#3}{#4}%
        \textstyle\scriptstyle\scriptstyle
  }{%
    \mint@@{#1}{#2}{#3}{#4}%
        \scriptstyle\scriptscriptstyle\scriptscriptstyle
  }{%
    \mint@@{#1}{#2}{#3}{#4}%
        \scriptscriptstyle\scriptscriptstyle\scriptscriptstyle
  }%
  \mkern-\thinmuskip
  \int#1%
  \ifx\\#3\\\else_{#3}\fi
  \ifx\\#4\\\else^{#4}\fi  
}
\newcommand*{\mint@@}[7]{%
  \begingroup
    \sbox0{$#5\int\m@th$}%
    \sbox2{$#5\int_{}\m@th$}%
    \dimen2=\wd0 %
    \let\mint@limits=#1\relax
    \ifx\mint@limits\relax
      \sbox4{$#5\int_{\kern1sp}^{\kern1sp}\m@th$}%
      \ifdim\wd4>\wd2 %
        \let\mint@limits=\nolimits
      \else
        \let\mint@limits=\limits
      \fi
    \fi
    \ifx\mint@limits\displaylimits
      \ifx#5\displaystyle
        \let\mint@limits=\limits
      \fi
    \fi
    \ifx\mint@limits\limits
      \sbox0{$#7#3\m@th$}%
      \sbox2{$#7#4\m@th$}%
      \ifdim\wd0>\dimen2 %
        \dimen2=\wd0 %
      \fi
      \ifdim\wd2>\dimen2 %
        \dimen2=\wd2 %
      \fi
    \fi
    \rlap{%
      $#5%
        \vcenter{%
          \hbox to\dimen2{%
            \hss
            $#6{#2}\m@th$%
            \hss
          }%
        }%
      $%
    }%
  \endgroup
}

\theoremstyle{plain}%default
\newtheorem{thm}{Theorem}

\theoremstyle{definition}

\theoremstyle{remark}

\begin{document}

\title[Centroid bodies]{Local fixed point results for centroid body operators }
%\thanks{Adviser: the name of your adviser}
%The Euclidean ball is locally the only fixed point for the $p$-centroid body operators
\author[C. Reuter]{Chase Reuter}

\address{Department of Mathematics 2750\\ North Dakota State University\\PO BOX 6050\\ Fargo, ND 58108-6050\\ USA}

\email{chase.reuter@ndsu.edu}


% Keywords


\begin{abstract}
We prove that, in a neighborhood of the Euclidean ball, there are no other fixed points of the $p$-centroid body operator, using spherical harmonic techniques. We also show that the Euclidean ball is locally the only body whose centroid body is a dilate of its polar intersection body.
\end{abstract}

\maketitle

\section{Introduction}
%History
Characterizing the Euclidean space among all real normed spaces is one of the aims of the ten problems formulated in 1956 by  Busemann and Petty  \cite{BP}.  These problems lead to the study of certain operators on convex bodies in $\bbrn$, such as the intersection body operator  for the first Busemann-Petty problem.  Proving a particular property of the intersection body operator in each dimension resulted in the resolution of the first Busemann-Petty problem in 1996.  In the class of convex bodies, proving properties for operators globally is challenging.  The local study of such operators on convex bodies appears to be a  more approachable initial step.  In \cite{FNRZ}, it was established that, other than linear transformations of the Euclidean ball,  there are no additional fixed points of the intersection body operator in a neighborhood of the Euclidean ball.  A similar result was obtained in \cite{SZ} for the projection body operator. In  \cite{ANRY}, a local affirmative answer was obtained for the Busemann-Petty problems 5 and 8, by studying the local fixed points of the polar-intersection body operator.

In this paper we conduct a similar local study for several problems involving the centroid body operator. First, we consider its fixed points.
\begin{problem}{1} \textit{
    If for an origin-symmetric convex body $K\subset\bbrn$, $\dims\geq 3$ we have
    \begin{equation}\label{eqn:prob1}
        c\Gamma K=K,
    \end{equation}
    where $\Gamma K$ is the centroid body of $K$ and $c$ is some constant, need $K$ be an ellipsoid?}
\end{problem}

Below we will verify that the ball does indeed satisfy this condition, and that equation \eqref{eqn:prob1} is invariant under linear transformations.  Consequently, this equation holds for all ellipsoids. We prove the following result.

\begin{thm}\label{c1}
    Let $\dims\geq 3$.  If an origin-symmetric  convex body $K\subset\bbrn$ satisfies \eqref{eqn:prob1}  and is sufficiently close to the Euclidean ball in the Banach-Mazur distance, then $K$ must be an ellipsoid.
\end{thm}

A similar result is obtained for the $p$-centroid body operators $\Gamma_p$, $p\in (1, +\infty)$.  

\begin{thm}\label{cp}
    Let $\dims\geq 3$.  If an origin-symmetric  convex body $K\subset\bbrn$ satisfies 
        \begin{equation}\label{eqn:prob2}
        c\Gamma_p K=K,
    \end{equation}
    and is sufficiently close to the Euclidean ball in the Banach-Mazur distance, then $K$ must be an ellipsoid.
\end{thm}

In addition, we prove a local result involving the $p$-centroid body and the polar intersection body of $K$:
\begin{thm}\label{ci}
    Let $\dims\geq 3$ and $p \geq 1$.  If an origin-symmetric  convex body $K\subset\bbrn$ satisfies 
        \begin{equation}\label{eqn:prob3}
        c\Gamma_p K=\cali^*K,
    \end{equation}
    and is sufficiently close to the Euclidean ball in the Banach-Mazur distance, then $K$ must be an ellipsoid.
\end{thm}

Theorem \ref{ci} is similar  to several open problems listed by Gardner in \cite[pg 336-337]{Gar}, relating the intersection body of $K$ with the difference body $\Delta K$ and the cross section body $CK$ of $K$:

\bigskip

{\bf Problem 8.1} Suppose that $K$ is a convex body such that $\Delta K$ and $\cali^*K$ are homothetic. Must $K$ be a centered ellipsoid?

\bigskip

{\bf Problem 8.9 (ii)} If $\Pi K$ is homothetic to either $\cali K$ or $CK$, is $K$ an ellipsoid?

\bigskip

In two dimensions, the answer to Problem 1 is known to be affirmative: the only  convex bodies which are dilates of their centroid bodies are ellipses. This was proven by Petty in \cite[Theorem 5.6]{Petty}. %Indeed,  condition \eqref{eqn:prob1} can be restated as follows: The Third Theorem of Dupin  in two dimensions states that the radius of curvature $R(\xi)$ of the centroid body  $\Gamma K$ at the centroid perpendicular to a unit direction $\xi$ is proportional to the cube of the length of the chord $\ell(\xi)=K\cap \xi^\perp$,
%\[
%R(\xi)=\frac{\ell(\xi)^3}{24},
%\]
 %If we assume that $\Gamma K$ is homothetic to $K$, then the same relation  is true between the curvature of $K$ at a point and the cube of the length of the corresponding central section. But this is a restatement of the Busemann-Petty Problem number 8  in dimension 2, \cite{BP}.
 Indeed,  the centroid body of $K$ has an important role in fluid statics. If $K$ is an origin-symmetric convex body with density 1/2 (with water having density 1), each point on the boundary of the centroid body of $K$ is the center of mass of the underwater part of $K$ in a given direction, and $\Gamma K$ is the convex hull of the locus of all such centers of mass in all directions. The flotation properties of $K$ are governed by three Theorems of Dupin for $\Gamma K$ (see for example \cite{Ry1}). In particular, in two dimensions, the Third Theorem of Dupin  states that the radius of curvature $R(\xi)$ of the centroid body  $\Gamma K$ at a boundary point perpendicular to a unit direction $\xi$ is proportional to the cube of the length of the chord $\ell(\xi)=K\cap \xi^\perp$,
\[
R(\xi)=\frac{\ell(\xi)^3}{24},
\]
 If we assume that $\Gamma K$ is homothetic to $K$, then the same relation  is true between the curvature of $K$ at a point and the cube of the length of the corresponding central section. But in the two dimensional case this is a restatement of the Busemann-Petty Problem number 8  in dimension 2 \cite{BP,Petty}.

\section{Analytic statement of the problems}
   
   Let $\calsn=\{x\in \mathbb{R}^n: |x|=1\}$ be the unit sphere in $\mathbb{R}^n$, where $|x|$ denotes the Euclidean norm. 
      Let $\sigma_{\dims-1}$ be the spherical Lebesgue measure on $\calsn$ and $\lambda_\dims$ be the $\dims$-dimensional Lebesgue measure. 
      A convex body $K$ is a convex compact subset of $\mathbb{R}^n$ with non-empty interior.   We will assume that the origin is an interior point of convex bodies.  $K$ is origin symmetric if for every $x\in K$, $-x \in K$.
      
      Given a convex body $K\subset\bbrn$,  the support function $h_K$ of $K$ is defined as  
      \[
      h_K(\theta)=\max\{x \cdot \theta\setdiv x\in K\},
      \]       
      and its radial function  is 
      \[
\rho_K(\theta)=\max\{t>0:\, t\theta\in K   \}.
\]

  %  \subsubsection{$c\Gamma K=K$}
       A body $L$ is called the centroid body of a convex body $K$ if
        \[h_{L}(\phi)=\frac1{\vol_\dims(K)}\int_K\abs{\phi\cdot x}\intdiv d\lambda_n(x), \hspace{10pt}\forall \phi\in\calsn.\]
        The body $L$ is denoted as $\Gamma K$.  Switching this integral to polar coordinates and simplifying yields
        \begin{align*}
            h_{\Gamma K}(\phi)&=\frac 1{\vol_\dims(K)}\int_{\calsn}\int_0^{\rho_K(\theta)}\abs{\phi\cdot(r\theta)}r^{\dims-1}\intdiv dr\intdiv d\sigma_{\dims-1}(\theta)\\
            &=\frac 1{(\dims+1)\vol_\dims(K)}\int_{\calsn}\abs{\phi\cdot\theta}\rho_{K}^{\dims+1}(\theta)\intdiv d\sigma_{\dims-1}(\theta).
        \end{align*}
        Examining the right hand side, we observe that it is a multiple of the cosine transform.  The cosine transform $\calc f$ of an integrable function $f:\calsn\rightarrow\bbr$  is defined as
        \[\calc f(\phi):=\int_{\calsn}\abs{\phi\cdot\theta}f(\theta)\intdiv d\sigma_{\dims-1}(\theta).\]
        Using this to rewrite the support function of the centroid body gives
        \begin{equation}\label{CentroidBodyReformulation}h_{\Gamma K}=\frac{1}{(\dims+1)\vol_\dims(K)}\calc\rho_{K}^{\dims+1}.\end{equation}
        Equation \eqref{eqn:prob1} can now be rewritten analytically  as 
        \begin{equation}\label{CentroidBody}
        h_{K}=\frac{c}{(\dims+1)\vol_\dims(K)}\calc\rho_{K}^{\dims+1}.
        \end{equation}

%{\color{red} Define $p$ centroid body here or say that it will be introduced in Section... }

%\subsubsection{$\Gamma K=\cali^* K$}
%Now we will derive the analytic reformulation of \eqref{eqn:prob3}.

A body $L$ is called the intersection body of the convex body $K$ if
        \[h_{L}(\theta)=\vol_\dims(K\cap\theta^\perp), \hspace{10pt}\forall \theta\in\calsn.\]
        The body $L$ is denoted as $\cali K$.  Expressing the volume as an integral and passing to polar coordinates,  we obtain
        \begin{align*}
            \rho_{\cali K}(\phi)&=\int_{K\cap \phi^\perp}\intdiv d\lambda_{n-1}(x)\\
                &=\int_{\calsn\cap \phi^\perp}\int_0^{\rho_K(\theta)}r^{\dims-2}\intdiv dr\intdiv d\sigma_{\dims-2}(\theta)\\
                &=\frac1{\dims-1}\int_{\calsn\cap \phi^\perp}\rho_K^{\dims-1}(\theta)\intdiv d\sigma_{\dims-2}(\theta).
        \end{align*}
        The right hand side is almost the spherical Radon transform.  The spherical Radon transform of an integrable function $f\in L^1(\calsn)$ is denoted $\calr f$ and is defined as
        \[\calr f(\phi)=\int_{\calsn\cap\phi^\perp}f(\theta)\intdiv d\sigma_{\dims-2}(\theta).\]
        Thus, the radial function of the intersection body can be expressed in terms of the Radon transform of the radial function of the original body,  by
        \begin{equation}\label{IntersectionBodyReformulation}\rho_{\cali K}=\frac 1{\dims-1}\calr \rho_K^{\dims-1}.\end{equation}
        A body $L$ is the polar body of a convex body $K$ if
        \[h_L(\theta)=\frac1{\rho_K(\theta)}, \forall\theta\in\calsn.\]
        The body $L$ is denoted by $K^*$. The polar intersection body of a body $K$ is $\cali^*K=(\cali K)^*$.  Hence, 
        \[h_{\cali^* K}=(\dims-1)(\calr \rho_K^{\dims-1})^{-1}.\]
        Combining this with  equation \eqref{CentroidBodyReformulation} gives the reformulation of $c\Gamma K=\cali^* K$ as the following equation,
        \begin{equation}\label{IStarReformulation}\frac c{(\dims^2-1)\vol_\dims(K)}\calc\rho_K^{\dims+1}=(\calr \rho_K^{\dims-1})^{-1}.\end{equation}
To simplify the exposition, we will first consider the case $p=1$. The analytic reformulation of equation \eqref{eqn:prob2} will be derived in Section \ref{five}, and the analytic reformulation of \eqref{eqn:prob3} will be considered in Section \ref{six}.

\bigskip

\textbf{Remark:} Even though Theorem \ref{c1}  does not make such assumptions on the body $K$, if $K$ satisfies \eqref{eqn:prob1}, then it must be centrally symmetric, strictly convex and smooth, since $\Gamma K$ has those properties (see \cite[Sec. 9.1]{Gar}).  

%{\color{red} How about Theorems 2 and 3? }

%[[Remark on the necessity of the conditions on the body: the centroid body operator smoothens things, so our $K$ satifying (1) will be origin symmetric, strictly convex and smooth.  ]]

\newpage
        
\subsection{Ellipsoids as solutions}

\subsubsection{$c\Gamma K=K$}
 Let us denote by $C_K$ the constant $c$ in \eqref{CentroidBody} corresponding to a body $K$ that satisfies \eqref{eqn:prob1}. If we evaluate \eqref{CentroidBody} when $K$ is the unit Euclidean ball $B$, we obtain
    \[
       1= h_B(\theta)=\frac{C_B}{(\dims+1)\vol_{\dims}(B)}\calc 1
            =\frac{2C_B\kappa_{\dims-1}}{(\dims+1)\kappa_\dims},
            \]
            since $\calc 1=2\kappa_{\dims-1}$, where $\kappa_n$ represents the volume of the unit Euclidean ball in $\mathbb{R}^n$. 
            We conclude that 
            \[
       % \implies
        C_B=\frac{(\dims+1)\kappa_{\dims}}{2\kappa_{\dims-1}}.
        \]
   % \end{align*}
%where $\kappa_\dims$ is the volume of the unit Euclidean ball in $\mathbb{R}^n$.
Thus, $B$ satisfies equation \eqref{eqn:prob1} with this constant $C_B$.

Next, we will verify that \eqref{CentroidBody} is invariant under linear transformations of the body $K$. This will imply  that ellipsoids are also solutions of \eqref{eqn:prob1}.  For $T\in\gln$, we have
        \begin{align*}
        	h_{\Gamma TK}\left(\phi\right)&=\frac 1{\vol_n(TK)}\int_{TK}\abs{{\phi}\cdot{x}} dx\\
            	&=\frac {\abs{\det T^{-1}}}{\vol_n(K)}\int_{K}\abs{{\phi}\cdot{Tx}} \abs{\det T}dx&&x\rightarrow Tx\\
            	&=\frac 1{\vol_n(K)}\int_{K}\abs{(T^{*}\phi)\cdot x} dx\\
                &=h_{\Gamma K}(T^*\phi)=h_{T\Gamma K}(\phi).
        \end{align*}
Since linear transformations commute with the centroid body operator, if $K$ satisfies $K=C_K\Gamma K$, then so does %$TK=Tc\Gamma K=c\Gamma TK$, and hence 
$TK$, with the same constant $C_{TK}=C_K$.

\subsubsection{$c\calip K=\Gamma  K$}

Let $C_{K}'$ be the constant associated with a body $K$ that satisfies the condition of Theorem \ref{ci}.  If we evaluate \eqref{IStarReformulation} when $K$ is the Euclidean ball, we see
\begin{align*}
    (\calr1)&=C_B'(\dims-1)\left(\frac{\calc 1}{{(\dims+1)\kappa_\dims}}\right)^{-1},\\
    C_B'&=\left(\frac{\calr1}{\dims-1}\right) \left(\frac{\calc 1}{(\dims+1)\kappa_\dims}\right). 
\end{align*}
To show that all ellipsoids are fixed points, we need to check how linear transformations interact with \eqref{ci}.  Using the calculations found in Gardner's book \cite[pgs 21 \& 308]{Gar}, the equation below follows.
    \begin{align*}
        \Gamma  TK&=T\Gamma  K=TC_K'(\cali K)^*=C_K'(T^{-*}\cali K)^*\\&=C_K'(\abs{\det(T)}^{-1}\cali TK)^*=C_K'\abs{\det T}\cali^* TK. 
    \end{align*}
    Therefore if $K$ satisfies $c\Gamma K=\cali^* K$, so does $TK$, albeit for a different constant, $C_{TK}'=C_K'\abs{\det T}$. 

\newpage
    
\subsection{From the Banach-Mazur distance to  the Hausdorff distance}

\subsubsection{$c\Gamma K=K$}
     The Banach-Mazur distance (see \cite[pg 589]{Sch}) between two origin symmetric convex bodies $K$ and $L$, is defined as
    \[d_{BM}(K,L):=\min \left\{\frac{1+\lambda}{1-\lambda}\geq 1\setdiv (1-\lambda)K\subseteq TL\subseteq (1+\lambda )K, T\in\gln\right\}.\]

     Let $\epsilon>0$ and suppose that $d_{BM}(B,K)<\frac{1+\epsilon}{1-\epsilon}$. Taking the linear transformation $T$  where  the minimum is achieved yields 
    \[(1-\epsilon)B\subseteq TK\subseteq (1+\epsilon)B.\]
    Then, for  the radial and support functions of $TK$ we have
    \begin{equation}\label{CentroidConstantBM}
        1-\epsilon\leq \rho_{TK}\leq 1+\epsilon\hspace{40pt}1-\epsilon\leq h_{TK}\leq 1+\epsilon,
    \end{equation}
    and taking volumes on each side gives
    \begin{equation}\label{CentroidConstantBMVol}
        (1-\epsilon)^\dims\leq \frac{\vol_\dims(TK)}{\kappa_\dims}\leq(1+\epsilon)^\dims.
    \end{equation}
    While in similar problems (see, for example,  \cite{ANRY,FNRZ}), an appropriate dilation of the body can be chosen so that $C_K=C_B$, here we have that the constant in \eqref{eqn:prob1} is invariant under linear transformations, making this  choice impossible. Therefore, we will use equations \eqref{CentroidConstantBM} and \eqref{CentroidConstantBMVol} to obtain a bound on the constant $C_{TK}=C_K$ in terms of $C_B$ and $\epsilon$. 
    
    The cosine transform is a positive operator and as such it preserves inequalities.  Applying the cosine transform to all sides of the radial function estimate \eqref{CentroidConstantBM}, and adjusting the constants, we obtain 
\[
    \calc( 1-\epsilon)^{n+1} \leq %\frac{1}{(\dims+1)\vol_\dims(TK)}
     \calc\rho_{TK}^{\dims+1} \leq \calc (1+\epsilon)^{n+1},  
\]
\[
   \frac{C_B}{(n+1) \kappa_n} \calc( 1-\epsilon)^{n+1} \leq \frac{C_B}{(n+1) \kappa_n} 
     \calc\rho_{TK}^{\dims+1} \leq \frac{C_B}{(n+1) \kappa_n} \calc (1+\epsilon)^{n+1}. 
\]
Using equation \eqref{CentroidBodyReformulation}, we can rewrite all three terms in the inequality as 
\[
 (1-\epsilon)^{n+1} \leq \frac{C_B}{ \kappa_n}\vol_\dims(TK)
   \,  h_{\Gamma (TK)} \leq (1+\epsilon)^{n+1},
\]
    
    \[(1-\epsilon)^{\dims+1}\leq \frac{C_B\vol_\dims(TK)}{C_{K}\kappa_\dims}h_{C_{K}\Gamma TK}\leq (1+\epsilon)^{\dims+1}.\]
    
    Therefore, if $TK=C_K \Gamma (TK)$, we have
    \[(1-\epsilon)^{\dims+1}\leq \frac{C_B}{C_{K}}\frac{\vol_\dims(TK)}{\kappa_\dims}h_{TK}\leq (1+\epsilon)^{\dims+1}.\]
    Now we use the support function estimate in \eqref{CentroidConstantBM},  together with \eqref{CentroidConstantBMVol}, to obtain 
    \[\frac{(1-\epsilon)^{\dims+1}}{(1+\epsilon)^{\dims+1}}\leq \frac{C_B}{C_{K}}\leq \frac{(1+\epsilon)^{\dims+1}}{(1-\epsilon)^{\dims+1}}.\]
    Therefore,  $C_{K}=C_B+O(\epsilon)$.  

\bigskip

Among all linear transformations, we will choose one that places $K$ in the \textit{ isotropic position, i.e.}, the position where
\begin{equation}\label{isot} 
\int_K\abs{x\cdot y}^2dy=c\,|x|^2\qquad\qquad \forall x\in {\mathbb R^n}.
\end{equation}
It should be noted that the isotropic position as stated above is not unique as any dilate of $K$ must also satisfy \eqref{isot}.  We are interested in convex bodies close to the Euclidean ball and will accordingly select a dilate where $\vol_\dims(K)\approx \kappa_\dims$ per \eqref{CentroidConstantBMVol}.  See, for example  \cite[Sec. 2.3.2]{BGVV} for more details about the derivation of the isotropic position.   In  \cite[Sec. 5]{ANRY}, it is shown that if $\epsilon>0$ and $(1-\epsilon)B\subseteq K\subseteq (1+\epsilon)B$, then 
for the isotropic position $K'$ of $K$ with $\vol_\dims(K)=\vol_\dims(K')$, we have 
    %Since the constant in \eqref{eqn:prob1} is invariant under linear transforms, if we apply a linear transform $S$ to $TK$ to put it in the isotropic position (see \cite[pg.5]{ANRY}), we also have $C_{STK}=C_{K}=C_B+O(\epsilon)$, and
    \[(1-\epsilon)\left(\frac{1-\epsilon}{1+\epsilon}\right)^{\frac{\dims+2}{2}}B\subseteq K'\subset(1+\epsilon)\left(\frac{1+\epsilon}{1-\epsilon}\right)^{\frac{\dims+2}{2}}B,
    \]
and the above considerations guarantee that the constant still satisfies $C_{K'}=C_B+O(\epsilon)$.

\subsubsection{$c\calip K=\Gamma K$} Starting with \eqref{CentroidConstantBM}, we perform  similar estimates for the cosine and Radon transform, obtaining %we see that for the right and left sides of $c\Gamma^*K=\cali K$ we get
\[
    \frac{1}{(1+\epsilon)^{\dims+1}}\leq \left(\frac{\calc\rho_K^{\dims+1}}{(\dims+1)\vol_n(K)}\right)^{-1}\left(\frac{\calc 1}{(\dims+1)\vol_\dims(K)}\right)
    \leq\frac{1}{(1-\epsilon)^{\dims+1}},
\]
\[
    (1-\epsilon)^{\dims-1}\leq\left(\frac{\calr\rho_K^{\dims-1}}{\dims-1}\right)\left(\frac{\dims-1}{\calr1}\right)\leq(1+\epsilon)^{\dims-1}.
\]
Taking the ratio of the two terms above, multiplying by \eqref{CentroidConstantBMVol} and using equation \eqref{IStarReformulation}  yields
\[
    \frac{(1-\epsilon)^{\dims}}{(1+\epsilon)^{\dims+1}}\leq\frac{C_B'}{C_K'}\leq\frac{(1+\epsilon)^{\dims} }{(1-\epsilon)^{\dims+1}}.
\]
Hence, the ratio of the constants is of the order $1+O(\epsilon)$.  Incorporating the isotropic position we still have $C_K'=C_B'+O(\epsilon)$.

\section{Spherical harmonics}
    Spherical harmonics are homogeneous harmonic polynomials restricted to the sphere.  They form the eigenspaces for a collection of common linear operators, such as the spherical Radon transform, the cosine transform and the Laplace-Baltrami operator on the sphere.  In this section, we will mostly follow  the expositions from  Groemer \cite{Gro} and Atkison-Han \cite{AH}.

    The space of all harmonic, homogeneous polynomials of degree $\stdindex$ on $\dims$ variables whose domain is restricted to the sphere is denoted $\calhnd$.  Its elements are called the spherical harmonics of degree $\stdindex$.  The spherical harmonic spaces are orthogonal with respect to the $L^2(\calsn)$ inner product. The space of all finite sums of spherical harmonics on $\dims$ variables is denoted by $\calhn=\oplus_{\stdindex=0}^\infty \calhnd$, and called the space of spherical harmonics.

    The space of spherical harmonics is dense in $L^2(\calsn)$ and by extension also in $C(\calsn)$.  The orthogonal projection onto the spherical harmonics of degree $\stdindex$, will be denoted $\calp[\stdindex]:C(\calsn)\rightarrow\calhnd$ (see \cite[Def. 2.11]{AH} for  an explicit definition of the projection).  The eigenvalues of the cosine transform on each of the spherical harmonic spaces are (see \cite[Lemmas 3.4.1, 3.4.5)]{Gro}), for $k=0$, $\calc\calp[0]=2\kappa_{\dims-1}\calp[0]$ and for even $\stdindex>0$
    \begin{align*}
        \calc\calp[\stdindex]&=(-1)^{(\stdindex-2)/2}2\kappa_{\dims-1}\frac{(\stdindex-3)!!(\dims-1)!!}{(\stdindex+\dims-1)!!}\calp[\stdindex]=\mu_{\dims,\stdindex}\calp[\stdindex].
      %  \calr\calp[\stdindex]&=(-1)^{\stdindex/2}\omega_{\dims-1}\frac{(\stdindex-1)!!(\dims-3)!!}{(\stdindex+\dims-3)!!}\calp[\stdindex]=\nu_{\dims,\stdindex}\calp[\stdindex],
    \end{align*}
    where $(\stdindex+2)!!=(\stdindex+2)(\stdindex!!)$.
     Since  the cosine transform is defined as an integral over a symmetric domain, $\calc\calp[\stdindex]=0$ when $\stdindex$ is odd.

When the body $K$ is placed in the isotropic position, we have the following estimate of the second harmonic of $\rho_K$, obtained in \cite[Sec. 9]{ANRY}. Let $r_0=\calp[0]\rho_K=\int_{\calsn} \rho_K d\sigma$ be the mean value of $\rho_K$ on the sphere, and let
$\rho_K-r_{\scriptscriptstyle{0}}=\calp[2](\rho_K)+\calp[4](\rho_K)+\dots$ be the spherical harmonic decomposition of 
$\rho_K-r_{\scriptscriptstyle{0}}$. From the definition \eqref{isot} of the isotropic position, we have that if $p(x)$ is a quadratic harmonic polynomial, $p(x)=\sum\limits_{i,j}a_{ij}x_ix_j$ with $\sum\limits_{i=1}^na_{ii}=0$, then 
\[
0=\int\limits_Kp(x)dx=c_n\int\limits_{\calsn}\rho_K^{n+2}(x)p(x)d\sigma(x). 
\]
This means $\rho_K^{n+2}$ has no second order term in its spherical harmonic decomposition.  If $K$, additionally, was also close to the Euclidean ball in the sense $\abs{\rho_K-r_0}<\epsilon$, then $\calp[2]\rho_K$ should also be small in the $L^2(\calsn)$ sense.  Indeed, estimating the error of the first order Taylor polynomial of $(x+r_0)^{\dims+2}$ at $\rho_K-r_0$, yields 
\[
|\rho_K^{n+2}-(
r_{\scriptscriptstyle{0}}^{n+2}+(n+2)r_{\scriptscriptstyle{0}}^{n+1}(\rho_K-r_{\scriptscriptstyle{0}}))|\le C\epsilon |\rho_K-r_{\scriptscriptstyle{0}}|.
\]
The second harmonic $\calp[2](\rho_K)$ of $\rho_K$ can now be estimated from $\rho_K^{n+2}$ using the previous equation.  Applying the $L^2(\calsn)$ to each side and using the orthogonality of the spherical harmonic spaces, we obtain 
 \[
(n+2)r_{\scriptscriptstyle{0}}^{n+1}\|\calp[2](\rho_K)\|_{L^2(S^{n-1})}\le C\epsilon \|\rho_K-r_{\scriptscriptstyle{0}}\|_{L^2(S^{n-1})},
\]
which yields that 
\begin{equation}\label{p2r}    
\|\calp[2](\rho_K)\|_{L^2(S^{n-1})}\le C'\epsilon \|\rho_K-r_{\scriptscriptstyle{0}}\|_{L^2(S^{n-1})}.
\end{equation}

\bigskip
     
        \subsection{Linear approximation of $\calc\rho_K^{\dims+1}$}\label{lincos}
        
        In \cite[Sec. 9]{ANRY},  an estimate on the linear component of the operator $(\calr [\rho_K]^\alpha)^\beta$ is obtained. 
        In this subsection we derive a similar estimate for the operator $[\calc\rho_K^\alpha]^\beta$ that appears in our problem.  %The goal for this section is obtaining a similar result for the cosine transform, and to cite the necessary results for the proof of Problem 1.

        Let $c,\gamma\in \bbr\setminus\{0\}$, we know that the first order Taylor polynomial of $x^\gamma$ at $c$ is
       \begin{equation}\label{TaylorExpansion}
           x^\gamma\approx c^\gamma+\gamma c^{\gamma-1}(x-c).
       \end{equation}
       Take $r_0=\calp[0]\rho_K$ be the constant $c$, $\gamma=\alpha$, and $\rho_K$ be $x$.  We can then estimate the error of the first order approximation of $\rho_K^\alpha$ when  $1-\epsilon<\rho_K<1+\epsilon$ and $\epsilon>0$ is small.  In particular, we see the Lipschitz constant $C_\alpha$ appear in the error bound, 
            \[
            \abs{\rho_K^\alpha-(r_0^\alpha+\alpha r_0^{\alpha-1}(\rho_K-r_0))}\leq C_\alpha \epsilon\abs{\rho_K-r_0}.
            \]
		Observe that for any function $f\geq 0$, $\calc f\geq0$.  That is $\calc$ is a positive operator in the vector space sense.  Using the fact that $\calc$ is a positive linear operator,
            \[
            \calc\abs{\rho_K^\alpha-(r_0^\alpha+\alpha r_0^{\alpha-1}(\rho_K-r_0))}\leq C_\alpha \epsilon\,\calc\abs{\rho_K-r_0}.
            \]
        As positive linear operators admit a triangle inequality $\abs{\calc f}\leq\calc\abs{f}$, we get
            \[
            \abs{\calc\rho_K^\alpha-(\calc r_0^\alpha+\alpha r_0^{\alpha-1}\calc(\rho_K-r_0))}\leq C_\alpha \epsilon  \,\calc\abs{\rho_K-r_0}.
            \]
		Finally, let us examine the first order Taylor polynomial \eqref{TaylorExpansion} again but with $x=\calc\rho_K^\alpha$, $c=\calc r_0^\alpha$, and $\gamma=\beta$.
        \begin{align*}
            (\calc\rho_K^\alpha)^\beta&\approx (\calc r_0^\alpha)^\beta+\beta(\calc r_0^\alpha)^{\beta-1}\left[\calc (\rho_K^\alpha-r_0^\alpha)\right]\\
            &\approx (\calc r_0^\alpha)^\beta+\beta(\calc r_0^\alpha)^{\beta-1}\left[\alpha r_0^{\alpha-1}\calc (\rho_K-r_0)\right]
        \end{align*}
        Estimating the error of the approximation, we then see
            \begin{equation}\label{LinearizationError}
                \abs{(\calc\rho_K^\alpha)^\beta-(r_0^{\alpha\beta}(\calc 1)^{\beta}+\alpha\beta (\calc 1)^{\beta-1}r_0^{\alpha\beta-1}\calc(\rho_K-r_0))}\leq C_{\alpha,\beta} \, \epsilon \, \calc\abs{\rho_K-r_0},
            \end{equation}
        where the Lipschitz constant $C_{\alpha,\beta}$ is dependent on the value of $\calc 1$, $\alpha$, $\beta$, and the maximum value of $\epsilon$ allowed.  Notably, this is bounded   Note that in the estimate of \eqref{LinearizationError} only the positivity of the linear operator $\calc$ was used.  Considering some other positive linear operator $\calm$ (positive in the $\calm f\geq 0$ when $f\geq 0$ sense), the particular constant $C_{\alpha,\beta}$ will change but an identical estimate holds for $[\calm \rho_K^\alpha]^\beta$.  
      
        The next lemma requires a definition.  Let $\calm$ be a bounded linear operator on $L^2(\calsn)$ whose eigenspaces are precisely $\calhnd$, that is, there exist eigenvalues $\mu_\stdindex$ such that for all $f\in L^2$,
        \[\calm f=\sum_{\stdindex\geq 0}\mu_\stdindex\calp[\stdindex] f.\]
        Then $\calm$ is said to be a strong contraction if
        \[\max_{\stdindex\geq 0}\abs{\mu_\stdindex}<1,\hspace{20pt}\text{ and }\hspace{20pt}\lim_{\stdindex\rightarrow \inf}\mu_\stdindex=0.\]
        It should be noted that small enough constant multiples of the cosine and Radon transforms are strong contractions, for example. The main result about the strong contraction is  \cite[Lemma 4]{ANRY}, stated below. %, which gives as a conclusion is a characterization of the ball from an inequality between the support and radial functions of a body.

        \begin{lemma}{4}\textit{  Assume $\calm$ is a strong contraction.  Then there exists $\epsilon\in(0,1)$ such that for any symmetric convex body $K$ and any $c\in (1-\epsilon,1+\epsilon)$, the conditions
        \[1-\epsilon\leq\rho_K\leq1+\epsilon,\hspace{20pt}\norm[L^2]{(h_K-h_0)-c\calm(\rho_K-r_0)}\leq\epsilon\norm[L^2]{\rho_K-r_0},\]
        imply $h_K=\rho_K=$const.  Here $h_0$ and $r_0$ are the constant terms of the spherical harmonic decomposition.}
        \end{lemma}

\section{Proof of Theorem 1}
   
    \begin{theorem}{1} \textit{There exists an $\epsilon_2>0$, such that the only convex bodies $K$ satisfying \eqref{eqn:prob1} with $d_{BM}(B,K)<\epsilon_1$ are ellipsoids.}
    \end{theorem}
    \begin{proof}
        This will be proven in two steps, first under certain assumptions on $K$ and then in general. First, we will assume that $K$ satisfies \eqref{eqn:prob1}, is close to the sphere (in the sense that  $1-\epsilon\leq\rho_K\leq1+\epsilon$), and is placed in the isotropic position with $\vol_\dims(K)=\kappa_\dims$. 
        
         Using the same approach as in the papers \cite{ANRY,FNRZ,SZ}, we will separate the operator into linear and non-linear components and examine each separately. %Due to the invariance of the problem under linear transformations, the body $K$ may be assumed to have the volume of the unit ball $\kappa_n$ and be in the isotropic position. Hence, 
         With our assumptions on $K$, equation \eqref{CentroidBody} reads as 
        \[
            h_K=\frac {C_K}{(\dims+1)\kappa_{\dims}}\calc\rho_K^{\dims+1}.
        \]
        As we saw in Section 3, we write $\rho_K=r_0+\gamma$, where $r_0$ is the zero harmonic of $\rho_K$. Since  $\calc$ is a linear operator and constants are eigenvectors of it, we have
        \begin{equation}\label{Problem1Linearization}
            h_K=\frac {C_K}{(\dims+1)\kappa_{\dims}}\calc(r_0+\gamma)^{\dims+1}=\frac{C_K}{C_B}r_0^{\dims+1}+\frac {C_K r_0^{\dims}}{\kappa_{\dims}}\calc \gamma +G_1,
        \end{equation}
%                &=a^{\dims+1}+\frac {a^{\dims}(\dims+1)}{2\kappa_{\dims-1}}\calc \gamma +\sum_{j=2}^{\dims+1}\binom{\dims+1}{j}ca^{\dims-j}\calc\gamma^j
        where $G_1$ contains all  the higher order terms.  Thus, the linear component may be rewritten as,
        \begin{equation}\label{linest}
                    r_0^{\dims}\frac{C_K}{\kappa_{\dims}}\calc \gamma=\left(r_0^{\dims}\frac{C_K}{C_B}\right)\left(\frac{C_B}{\kappa_\dims}\calc\right)\gamma.
       \end{equation}
        % We observe that if $a=1+O(\epsilon)$, then the    constant  $\left(a^{\dims}\frac{C_K}{C_B}\right)$ is of the same order $1+O(\epsilon)$.  The operator term $\left(\frac{C_B}{\kappa_\dims}\calc\right)$ is a normalization of the cosine transform.  Finally, if $a=1+O(\epsilon)$, then $\left(\rho_K-a\right)$  is of the order $O(\epsilon)$.
        %Multiplying \eqref{Problem1Linearization} by   $\frac{C_{K}}{(\dims+1)\kappa_{\dims}}$, taking the $L^2$ norms on both sides, and using  \eqref{LinearizationError} with $\alpha=n+1,\beta=1$ we obtain
        %{\color{red} write first line with $h_K$ and what really relates to  (8) and (10)}
         %  \[
           %     \norm[L^2]{\frac{C_K}{(\dims+1)\kappa_\dims}\calc\rho_K^{\dims+1}-\frac{C_K\calc1}{(\dims+1)\kappa_\dims}r_0^{\dims+1}-r_0^{\dims}\frac{(\dims+1)C_K}{(\dims+1)\kappa_\dims}\calc(\rho_K-r_0)}.
             %   \]
        Estimating the higher order terms, we solve for $G_1$ in \eqref{Problem1Linearization} and take $L^2$ norms yielding
         \begin{align*}
            \norm[L^2]{G_1} = \norm[L^2]{\left[h_{K}-r_0^{\dims+1}\frac{C_K}{C_B}\right]-\left(r_0^{\dims}\frac{C_K}{C_B}\right)\left(\frac{C_B}{\kappa_n}\calc\right)\gamma}.
            \end{align*}
             Using \eqref{LinearizationError} with $\alpha=n+1,\beta=1$ we obtain 
        \begin{align*}
            \norm[L^2]{G_1}&\leq C_{\dims+1,1}\frac{C_K}{C_B}\epsilon\norm[L^2]{\frac{C_B}{(\dims+1)\kappa_\dims}\calc\abs{\gamma}}\\
            &\leq C_{\dims+1,1}\frac{C_K}{C_B}\epsilon\norm[L^2]{\gamma}.
        \end{align*}
        In the notation of Lemma 4, we have $h_0=r_0^{\dims+1}C_KC_B^{-1}$, $c=r_0^{\dims}C_KC_B^{-1}$, and $\calm=C_B\kappa_\dims^{-1}(\calc-\calc\calp[0]-\calc\calp[2])$. 
        The  degree zero harmonic of $\rho_K-r_0$ is zero, and the second harmonic of $\gamma$ has already been estimated by \eqref{p2r}. The eigenvalues of operator $\calm$ are scalar multiples of the eigenvalues of the cosine transform starting with degree 4.  Since the even degree eigenvalues decay  monotonically to zero in absolute value, we only need to check that the first non-zero eigenvalue is less than one in absolute value.  Indeed, $\calm\calp[4]=-(\dims+3)^{-1}\calp[4]$, and  $\calm$ is a strong contraction as desired.
        By lemma 4, there is some $0<\epsilon_0$ where if $K$ satisfies,
        \begin{itemize} \item $1-\epsilon_0\leq \rho_K\leq 1+\epsilon_0$,
            \item $c\in (1-\epsilon_0,1+\epsilon_0)$,
            \item $\norm[L^2]{(h_K-h_0)-c\calm(\rho_K-r_0)}\leq\epsilon_0\norm[L^2]{\rho_K-r_0}$,
        \end{itemize}
        $K$ must be the Euclidean ball.
        
        First, we note that the constant $c$ is bounded by the product of the bounds on $r_0^{\dims}$ and the bounds on the ratio $C_B/C_K$, giving $c=1+O(\epsilon)$.  If $\epsilon$ is sufficiently small $1-\epsilon_0<c<1+\epsilon_0$ follows.  The second criteria follows if $\epsilon<\epsilon_0$.  Checking the third criteria, 
            \[\norm[L^2]{\left[h _K-h_0\right]-c\calm\left(\rho_K-r_0\right)}\]
            \[\leq \norm[L^2]{\left[h_K-h_0\right]-c\left(\frac{C_B}{\kappa_\dims}\calc\right)\gamma}+c\norm[L^2]{\frac{C_B}{\kappa_\dims}\calc\calp[2]\gamma}\]
                \[\leq \bigg(C_{\dims+1,1}\frac{C_K}{C_B}+c\frac{C_{\dims+2}}{r_0^{\dims+1}}\bigg)\epsilon\norm[L^2]{\gamma}.\]
        Now we need to estimate the expression in parentheses.  We have that
        \[C_{\dims+1,1}\frac{C_K}{C_B}+c\frac{C_{\dims+2}}{r_0^{\dims+1}}\leq (\dims+1)\frac{(1+\epsilon)^{2\dims+1}}{(1-\epsilon)^{\dims+1}}+(1+\epsilon_0)(\dims+2)\frac{(1+\epsilon)^{\dims+1}}{(1-\epsilon)^{\dims+1}}\]
        is bounded above. Hence, the product of this bound with $\epsilon$ is bounded above by $\epsilon_0$, provided $\epsilon$ is small enough.  
        
        Now we prove Theorem \ref{c1} for all bodies close to the ball.  Let $0<\epsilon_1$ be the maximal $\epsilon$ for which the additional criteria for Lemma 4 hold in the isotropic case.  Put
        \[\epsilon_2=\max\left\{\epsilon>0\biggsetdiv \abs{1-(1\pm\epsilon)\left(\frac{1\pm \epsilon}{1\mp \epsilon}\right)^{\frac{\dims+2}{2}}}\leq \epsilon_1\right\}.\]
        If $K$ is a convex body with $d_{BM}(B,K)\leq\epsilon_2$ and satisfying \eqref{eqn:prob1}, then there exists a linear transformation $T\in\gln$ where $TK$ satisfies the assumptions of the first part of the proof.  Therefore, $TK$ is a dilate of the Euclidean ball and $K$ is an ellipsoid.
    \end{proof}

%    Let $\epsilon_0$ be the minimum of the $\epsilon$ satisfying the constant $\abs{1-c}<\delta$, the right hand constant is less than delta, and $\epsilon<\delta$.  Then the only convex body $K$ with $d_{BM}(B,K)=m$ satisfying $\abs{1-(1\pm m)^{\dims+1}(1\mp m)^{-(\dims+1)}}<\epsilon_0$ and 

\section{Generalization to $p$-centroid bodies: Proof of Theorem \ref{cp}}\label{five}
    This section introduces $p$-centroid body and extends the result of the previous section to $p>1$.  Subsection \ref{cpdefinitionsection} gives the definition of the $p$-centroid body and an analytic reformulation in terms of the $p$-cosine transform.  As the differences between the prior case and the $p>1$ are purely technical, the subsequent subsections include relevant details for the extension of Theorem \ref{c1} to $p>1$ for completeness. 
    \subsection{Definition}\label{cpdefinitionsection}
    The $p$-centroid body of a star body, $\Gamma_p K$ is defined in a  similar way to the centroid body by,
    \[h_{\Gamma_p K}(\theta):=\left(\mint{-}_K\abs{\theta\cdot x}^{p}dx\right)^{\frac1p}=\frac1{\sqrt[p]{(\dims+p)\vol_{\dims}(K)}}\left(\calc^p\rho_K^{\dims+p}(\theta)\right)^{\frac1p},\hspace{20pt}\forall\theta\in\calsn
    \]
    where $\mint{-}$ denotes the integral average $\mint{-}_K\,dx=\frac1{\vol_\dims(K)}\int_K\,dx$ and $\calc^p$ is the $p$-cosine transform, 
    \[
      \calc^p(f):=\int_{\calsn}\abs{\phi\cdot\theta}^p f(\theta)\intdiv d\sigma_{\dims-1}(\theta).
    \]
     The $p$-centroid bodies were introduced in \cite{LYZ}. For $p\geq 1$, the function $h_{\Gamma_p K}$ is the support function of a convex body, so $\Gamma_p K$ is well defined.  In particular for $p=1$, this is the centroid body studied in the previous sections, and for $p=2$ this is the so called Legendre ellipsoid.  For $p=2$, the relation $C_K \Gamma_2 K=K$ trivially implies  that $K$ is an ellipsoid. 
     
     For $p>1$, $p\neq 2$, we will study if 
    \[C_K\Gamma_p K=K\]
    for some scalar $C_K$ implies that $K$ is an ellipsoid, when $K$ is close enough to the Euclidean ball in the Banach-Mazur distance. The analytic reformulation of the above equation is 
\begin{equation}\label{argp}
    \frac{C_K}{\sqrt[p]{(\dims+p)\vol_{\dims}(K)}}\left(\calc^p\rho_K^{\dims+p}(\theta)\right)^{\frac1p}=h_K(\theta).
\end{equation}
     
%Please write a couple of words about the difference of the following and the previous section. It looks like this section is technically more complicated but nothing is added to the ideas. Is it right?

    \subsection{Linear transformations and Stability}
     As in the $p=1$ case, equation \eqref{argp} is invariant under linear transformations.  If $T\in\gln$, then
    \[h_{\Gamma_p TK}(\theta)=\left(\mint{-}_{TK}\abs{\theta\cdot x}^p dx\right)^{\frac1p}
        =\left(\mint{-}_{K}\abs{T^*\theta\cdot x}^p dx\right)^{\frac1p}=h_{T\Gamma_p K}(\theta).\]
    Because the $p$-centroid body operator commutes with linear transformations, if $K$ satisfies \eqref{argp} then so does $TK$,  and $C_{TK}=C_K$.  Since the unit Euclidean ball $B$ satisfies  \eqref{argp}, we have that ellipsoids also satisfy this equation with the constant 
    \[1=\frac{C_B}{\sqrt[p]{(\dims+p)\kappa_\dims}}(\calc^p1)^{\frac1p},\hspace{20pt}\text{or equivalently,}\hspace{20pt}C_B=\frac{\sqrt[p]{(\dims+p)\kappa_\dims}}{(\calc^p 1)^{1/p}}.\]
    Since $C_{K}$ in  \eqref{argp} is invariant under linear transformations of the body, we cannot choose a particular value for this constant by picking an appropriate  linear transformation. On the other hand, we can obtain stability bounds on $C_K$ in same way as in the case $p=1$.  Let $K$ be close to the Euclidean ball, that is for some $\epsilon>0$,
    \begin{equation}\label{close}
        1-\epsilon\leq \rho_{K}\leq 1+\epsilon, \hspace{40pt}1-\epsilon\leq h_{K}\leq 1+\epsilon.
    \end{equation}
     This implies % volume of $K$ is then bounded by the volumes of balls of radius $1-\epsilon$ and $1+\epsilon$.
    \begin{equation}\label{volu}\frac1{(1+\epsilon)^{\frac{\dims}{p}}}\leq\sqrt[p]{\frac{\kappa_\dims}{\vol_\dims(K)}}\leq \frac1{(1-\epsilon)^{\frac{\dims}{p}}}.
    \end{equation}
    Applying the $p$-cosine transform to the estimate on the radial function in \eqref{close}, we get 
    \[\frac{C_B}{\sqrt[p]{(\dims+p)\kappa_\dims}}\left(\calc^p(1-\epsilon)^{\dims+p}\right)^{\frac1p}
    \leq \frac{C_B}{\sqrt[p]{(\dims+p)\kappa_\dims}}\left(\calc^p\rho_K^{\dims+p}\right)^{\frac1p}\]
    \[
    \leq \frac{C_B}{\sqrt[p]{(\dims+p)\kappa_\dims}}\left(\calc^p(1+\epsilon)^{\dims+p}\right)^{\frac1p}.\]
   Combining this estimate with \eqref{volu} and with equation \eqref{argp} we have    
    \[(1-\epsilon)^{\frac{\dims+p}p}
    \leq \frac{C_B}{C_K}\sqrt[p]{\frac{\vol_\dims(K)}{\kappa_\dims}}\,h_{K}
    \leq (1+\epsilon)^{\frac{\dims+p}p},\]
    \[\frac{(1-\epsilon)^{\frac{\dims}p+1}}{(1+\epsilon)^{\frac{\dims}{p}}}
    \leq \frac{C_B}{C_K}\, h_{K}
    \leq \frac{(1+\epsilon)^{\frac{\dims}p+1}}{(1-\epsilon)^{\frac{\dims}{p}}},\]
and, since $1-\epsilon\leq h_{K}\leq 1+\epsilon$,     
    \[\left(\frac{1-\epsilon}{1+\epsilon}\right)^{\frac{\dims}p+1}
    \leq \frac{C_B}{C_K}
    \leq \left(\frac{1+\epsilon}{1-\epsilon}\right)^{\frac{\dims}p+1}.\]
   
    \subsection{The eigenvalues of $\calc^p$}\label{cpeig}
    The eigenspaces of the $p$-cosine transform are, as in the case of the cosine transform, the spaces of spherical harmonics of degree $k$.  From \cite[Eq. (3.4)]{Rubin} we have that the eigenvalues $\calc^\alpha \calp[\stdindex]=m_{\alpha,\stdindex}\calp[\stdindex]$ for $p\neq 2,4,\ldots$ are equal to 
    \[m_{p,\stdindex}=\begin{cases}
        \frac{2(-1)^{\stdindex/2}\pi^{(\dims-1)/2}}{\sigma_{\dims-1}}\frac{\Gamma((\stdindex-p)/2)}{\Gamma(-p/2)}\frac{\Gamma((p+1)/2)}{\Gamma((\stdindex+\dims+p)/2)}  &: 2\mid \stdindex\\
        0   &:  2\nmid \stdindex
    \end{cases}.\]
When $p$ is an even integer and $k>p$, the eigenvalues are 0. The above formula for the eigenvalues $m_{p,k}$ can be rewritten so that it works also when $p$ is an even integer and $k\leq p$. (For the reader who wants to consult reference \cite{Rubin}, we note that in that paper the $p$-cosine transform is denoted by $M^{\alpha+1}$). 

Observe that $\abs{m_{p,2\stdindex+2}/m_{p,2\stdindex}}=\abs{2\stdindex-p}/(2\stdindex+\dims+p)$. Expressing the higher order non-zero harmonics as a telescoping product of the lower order ones, we see that
    \[\abs{m_{p,2k}}=m_{p,0}\prod_{\ell=0}^{k-1}\abs{\frac{m_{p,2\ell+2}}{m_{p,2\ell}}}=m_{p,0}\prod_{\ell=0}^{k-1}\abs{\frac{2\ell-p}{2\ell+\dims+p}},\]
    hence the non-zero eigenvalues are absolutely strictly decreasing as $k\rightarrow\infty$, and their limit is $0$. 
    The asymptotic behaviour as $p$ varies is similar, for $m_{p,0}\rightarrow 0$ as $p\rightarrow \infty$. % For our purposes we need $\stdindex=0,2,4$, calculating these
%    \begin{align*}
%        m_{p,0}&=\frac{2\pi^{(\dims-1)/2}\Gamma(p/2)}{\sigma_{\dims-1}\Gamma((\dims+p)/2)}\\
%        m_{p,2}&=\frac{-\pi^{(\dims-1)/2}}{\sigma_{\dims-1}}\frac{(2-p)\Gamma(p/2)}{\Gamma((2+\dims+p)/2)}  
%    \end{align*}
    \subsection{Linear approximation error}
    The $p$-cosine tranansform is a positive linear operator, therefore if we set $\alpha=\dims+p$, $\beta=1/p$ on the left hand side of \eqref{LinearizationError}, we have 
    \begin{align*}
        &\abs{(\calc^p\rho_K^{\dims+p})^{\frac1p}-r_0^{\frac{\dims+p}{p}}(\calc^p 1)^{\frac1p}-\frac{\dims+p}p (\calc^p 1)^{\frac{1-p}p}r_0^{\frac\dims p}\calc^p(\rho_K-r_0))}. 
%        &=\norm[L^2]{\left(h_{C_K\Gamma_p K}-(\dims+1)\kappa_\dims r_0^{\frac{\dims+1}p}\frac{C_K}{C_B}\right)-}\\
%        &=\bigg\rvert\hspace{-1pt}\bigg\lvert\left(h_{C_K\Gamma_p K}-r_0^{\frac{\dims+p}{p}}(\calc^p 1)^{\frac1p}\right)-\\
%        &-\left(r_0^{\frac \dims p}\frac{C_K}{C_B}\right)\left(\frac{(\dims+p)^{2-\frac2p}}{p}\cdot\kappa_\dims^{1-\frac2p}C_B^{2-p}\calc^p\right)(\rho_K-r_0)\bigg\rvert\hspace{-1pt}\bigg\lvert_{L^2}\\
%        &=\bigg\rvert\hspace{-1pt}\bigg\lvert\left(h_{K}-r_0^{\frac{\dims+p}{p}}(\calc^p 1)^{\frac1p}\right)-\\
%        &-\left(r_0^{\frac \dims p}\frac{C_K}{C_B}\right)\left(\frac{(\dims+p)^{2-\frac2p}}{p}\cdot\kappa_\dims^{1-\frac2p}C_B^{2-p}\calc^p\right)(\rho_K-r_0)\bigg\rvert\hspace{-1pt}\bigg\lvert_{L^2}\\
    \end{align*}
    Taking the $L^2$ norm of this term and multiplying by the missing constants in equation \eqref{argp}, the left hand side can be rewritten as 
    \begin{equation}\label{along}
        \bigg\rvert\hspace{-1pt}\bigg\lvert\left(h_{K}-r_0^{\frac{\dims+p}{p}}\frac{C_K}{C_B}\frac{C_B}{\sqrt[p]{(\dims+p)\kappa_\dims}}(\calc^p 1)^{\frac1p}\right)
        \end{equation}
        \[
        -\left(r_0^{\frac\dims p}\, \frac{\dims+p}{p}\frac{C_K}{\sqrt[p]{(\dims+p)\kappa_\dims}}\left[
        %{\color{red} \left(\frac{C_B}{\sqrt[p]{(\dims+p)\kappa_\dims}}\right)^{-1}\left(\frac{C_B}{\sqrt[p]{(\dims+p)\kappa_\dims}}\right) }
        (\calc^p 1)^{\frac1p}\right]^{1-p}\calc^p\right)(\rho_K-r_0)\bigg\rvert\hspace{-1pt}\bigg\lvert_{L^2}.
        \]
        Recalling that $C_B[(\dims+p)\kappa_\dims]^{-1/p}(\calc^p1)^{1/p}=1$, we can simplify  the last term, obtaining
    \[\left(r_0^{\frac\dims p}\frac{\dims+p}{p}\frac{C_K}{\sqrt[p]{(\dims+p)\kappa_\dims}}\left(\frac{C_B}{\sqrt[p]{(\dims+p)\kappa_\dims}}\right)^{p-1}\calc^p\right)(\rho_K-r_0) \]
    \[=\left(r_0^{\frac\dims p}\frac{C_K}{C_B}\right)\left(\frac{\dims+p}{p}\frac{C_B^p}{(\dims+p)\kappa_\dims}\calc^p\right)(\rho_K-r_0) \]
    \[=\left(r_0^{\frac\dims p}\frac{C_K}{C_B}\right)\left(\frac{C_B^p}{p \,\kappa_\dims}\calc^p\right)(\rho_K-r_0).\]
    Thus,   \eqref{along} simplifies to
    \begin{equation}\label{lhs}
            \norm[L^2]{\left(h_K-r_0^{\frac{\dims+p}p}\frac{C_K}{C_B}\right)-\left(r_0^{\frac\dims p}\frac{C_K}{C_B}\right)\left(\frac{C_B^p}{p\kappa_\dims}\calc^p\right)(\rho_K-r_0)}.
        \end{equation}
    Next, we will consider the right hand side of \eqref{LinearizationError}. After multiplying by $C_K[(\dims+p)\kappa_\dims]^{-1/p}$ and taking $L^2$ norms, it can be rewritten as 
    \[  C_{\dims+p,\frac1p}\,\calc^p1 \, \epsilon\left(\frac{C_K}{\sqrt[p]{ (\dims+p)\kappa_\dims}}\right)\norm[L^2]{\calc^p\abs{\rho_K-r_0}}\]
    \[=C_{\dims+p,\frac1p}\calc^p1\, \epsilon\frac{C_K}{C_B}\left(\frac{C_B}{\sqrt[p]{ (\dims+p)\kappa_\dims}}\right)^{1-p}\norm[L^2]{\frac{C_B^p}{(\dims+p)\kappa_\dims}\calc^p\abs{\rho_K-r_0}}\]
    \begin{equation}\label{rhs}
            =C_{\dims+p,\frac1p}\,\frac{C_K}{C_B}\left(\calc^p 1\right)^{2-\frac{1}p}\epsilon\norm[L^2]{\frac{C_B^p}{(\dims+p)\kappa_\dims}\calc^p\abs{\rho_K-r_0}}.
    \end{equation}
    Note  that the exponent $2-\frac{1}{p}$ on $\calc^p1$ is non-negative for $p\geq 1$ and that 
    \begin{align*}
        \calc^p1&=\frac{2\pi^{(\dims-1)/2}}{\sigma_{\dims-1}}\frac{\Gamma(\frac{p+1}{2})}{\Gamma(\frac{\dims+p}{2})}=\frac1{\sqrt\pi}\frac{\Gamma(\frac{\dims}{2})\Gamma(\frac{p+1}{2})}{\Gamma(\frac{\dims+p}{2})}.
    \end{align*}
    We claim that this expression is bounded above by $1$ when the dimension $n$ is greater than $2$.  This is obtained by using the property $z\Gamma(z)=\Gamma(z+1)$ several times to yield
    \[\calc^p1=\left[\frac{\Gamma\left(\frac\dims2-\left\lceil\frac n2\right\rceil+1\right)}{\sqrt\pi}\right]\left[\prod_{\ell=0}^{\left\lceil\frac n2\right\rceil-1}\left(\frac{\dims-2\ell}{\dims+p-2\ell}\right)\right]\left[\frac{\Gamma\left(\frac {p+1}2\right)}{\Gamma\left(\frac p2+\frac\dims2-\left\lceil\frac n2\right\rceil+1\right)}\right].\]
    Breaking up the three terms in the product, we note that the leftmost is bounded by $1$, the middle term is maximal when $p=1$ and $\dims=3$ (which are the minimal values of $p,n$), and the last term is also maximal when $p=1$ and $\dims$ odd (by the convexity of $\Gamma$), yielding
    \[\calc^p1\leq 1\cdot\frac{3}{4}\cdot\frac{1}{2}\cdot 1 <1.\]
    %Reiterating that the exponent of $\calc^p1$ is non-negative, and 
    Using this estimate, and the fact that the $L^2$ norm of the normalized $p$-cosine transform is 1, \eqref{rhs} is bounded by  
    \[
      C_{\dims+p,\frac1p}\epsilon\frac{C_K}{C_B}\norm[L^2]{\rho_K-r_0}.
    \]
Therefore, combining this with \eqref{lhs}, we have the following estimate that we will need to apply Lemma 4.
    \begin{equation}\label{toapplylemma}   
    \norm[L^2]{\left(h_K-r_0^{\frac{\dims+p} p}\frac{C_K}{C_B}\right)-\left(r_0^{\frac\dims p}\frac{C_K}{C_B}\right)\left(\frac{C_B^p}{p\kappa_\dims}\calc^p\right)(\rho_K-r_0)}
    \end{equation}
    \[
        \leq C_{\dims+p,\frac1p}\epsilon\frac{C_K}{C_B}\norm[L^2]{\rho_K-r_0}.
    \]
    \subsection{The proof of Theorem 2.}
    The last thing that needs to be checked is that the operator $\frac{C^p_B}{p\kappa_\dims}\calc^p$  in  equation \eqref{toapplylemma} is a strong contraction (note that the constant multiplying this operator $r_0^{\frac\dims p}\frac{C_K}{C_B}$ is of the order $1+O(\epsilon)$, as discussed in subsection 5.2). As mentioned in Section \ref{cpeig}, the eigenvalues of the $p$-cosine transform for the harmonic spaces of even degree greater than or equal to 4  are less than 1, strictly decreasing in absolute value, and their limit is $0$. Thus, it is enough to check the eigenvalue of the degree four harmonic space of the scaled $p$-cosine transform $C_B^p[p\kappa_\dims]^{-1}\calc^p$. The computation yields
    %\[\abs{\frac{\dims+p}{p}\frac{C_B^p}{(\dims+p)\kappa_\dims}m_{p,2}}=\abs{\frac{\dims+p}{p}\frac1{m_{p,0}}m_{p,0}\frac{m_{p,2}}{m_{p,0}}}=1\]
    \[\abs{\frac{\dims+p}{p}\frac{C_B^p}{(\dims+p)\kappa_\dims}m_{p,4}}=\abs{\frac{m_{p,4}}{m_{p,2}}}=\frac{\abs{2-p}}{\dims+2+p}<1\]
    %(Note that this is consistent with the $p=1$ calculation).  
    From this point, the proof of Theorem \ref{cp} is identical to that of Theorem \ref{c1}. 

    \section{On Theorem \ref{ci}}\label{six}
    
    \subsection{Invariance under linear transformations}
          Let $C_{K}'$ be the constant associated with a body $K$ that satisfies the condition of Theorem \ref{ci}.  If we evaluate \eqref{IStarReformulation} when $K$ is the Euclidean ball, we see
        \begin{align*}
            (\calr1)&=C_B'(\dims-1)\left(\frac{\calc^p1}{{(\dims+1)\kappa_\dims}}\right)^{-\frac1p},\\
            C_B'&=\left(\frac{\calr1}{\dims-1}\right) \left(\frac{\calc^p1}{(\dims+1)\kappa_\dims}\right)^{\frac1p}.
        \end{align*}
        We next show that if $K$ satisfies $c\Gamma K=\cali^* K$, so does $TK$. % To show that all ellipsoids are fixed points, we need to check how linear transformations interact with \eqref{ci}. 
            \begin{align*}
                \Gamma_p TK&=T\Gamma_p K=TC_K'(\cali K)^*=C_K'(T^{-*}\cali K)^*\\&=C_K'(\abs{\det(T)}^{-1}\cali TK)^*=C_K'\abs{\det T}\cali^* TK. 
            \end{align*}
            Therefore if $K$ satisfies $c\Gamma K=\cali^* K$, so does $TK$, albeit for a different constant, $C_{TK}'=C_K'\abs{\det T}$.
    
    \subsection{Stability of the constants.} Applying the normalized $p$-cosine transform and Radon transform to $1-\epsilon \leq \rho_K \leq 1+\epsilon$, we have 
    \[
        \frac{1}{(1+\epsilon)^{\frac{\dims+p}p}}\leq \left(\frac{\calc^p\rho_K^{\dims+p}}{(\dims+p)\vol_\dims(K)}\right)^{\frac{-1}p}\left(\frac{\calc^p1}{(\dims+p)\vol_\dims(K)}\right)^{\frac1p}\leq\frac{1}{(1-\epsilon)^{\frac{\dims+p}p}},
    \]
    and
    \[
        (1-\epsilon)^{\dims-1}\leq\left(\frac{\calr\rho_K^{\dims-1}}{\dims-1}\right)\left(\frac{\dims-1}{\calr1}\right)\leq(1+\epsilon)^{\dims-1}.
    \]
    If $K$ satisfies $c\Gamma^*K=\cali K$, taking the ratio of the two terms above and multiplying by \eqref{CentroidConstantBMVol} yields
    \[
        \frac{(1-\epsilon)^{\frac{\dims}{p}}}{(1+\epsilon)^{\frac{\dims+p}p}}\leq\frac{C_B'}{C_K'}\leq\frac{(1+\epsilon)^{\frac{\dims}{p}}}{(1-\epsilon)^{\frac{\dims+p}p}}.
    \]
    Thus, the ratio of the constants is of the order $1+O(\epsilon)$.  Applying an appropriate linear transformation to put $K$ in the isotropic position, we still have $C_K'=C_B'+O(\epsilon)$.
        
    \subsection{ Linearization estimate}
        From the calculation of the constant for the ball, we immediately have
        \begin{equation}\label{t3}
                \calr_1\rho_K^{\dims-1}=\frac{C_K'}{C_B'}\sqrt[p]{\frac{\kappa_\dims}{\vol_\dims(K)}}(\calc_1^p\rho_K^{\dims+p})^{\frac{-1}{p}},
        \end{equation}
        where $\calr_11=\calc_1^p1=1$ are the normalized Radon and $p$-Cosine transforms.  
        As we assume $1-\epsilon<\rho_K<1+\epsilon$, the ratio $\abs{\frac{\rho_K-1}{r_0}}<1$ for sufficiently small $\epsilon$.  Consequently, we have convergence of the binomial series expansion of $(r_0+(\rho_K-r_0))^{-\frac1p}$.  The binomial expansion of the right hand side yields,
        \begin{align*}
            \frac{C_K'}{C_B'}\sqrt[p]{\frac{\kappa_\dims}{\vol_\dims(K)}}\sum_{m=0}^\infty\binom{m+\frac1p-1}{m}(-1)^mr_0^{-(\dims+p)(\frac1p+m)}\left(\calc_1^p\rho_K^{\dims+p}-r_0^{\dims+p}\right)^m.
        \end{align*}
        Taking the inverse (normalized) Radon transform of both sides of \eqref{t3}, and separating the constant \eqref{calip_gammma_Linearization_constant}, linear \eqref{calip_gammma_Linearization_linear}, and higher order terms \eqref{calip_gammma_Linearization_HOTC}-\eqref{calip_gammma_Linearization_HOTC-1} yields
        \begin{align}\label{calip_gammma_Linearization_constant}
            &\rho_K-r_0+\frac{r_0}{\dims-1}-\frac{C_K'}{C_B'}\sqrt[p]{\frac{\kappa_\dims}{\vol_\dims(K)}}\frac{r_0^{1-n-\frac{\dims}p}}{\dims-1}\\
            \label{calip_gammma_Linearization_linear}
            &=-\left(\frac{C_K'}{C_B'}\sqrt[p]{\frac{\kappa_\dims}{\vol_\dims(K)}}r_0^{-n(1+\frac{1}p)}\right)\left(\frac{\dims+p}{p(\dims-1)}\calr_1^{-1}\circ\calc_1^p\right)\left(\rho_K-r_0\right)\\
            \label{calip_gammma_Linearization_HOTC}
            &+\frac{C_K'}{C_B'}\sqrt[p]{\frac{\kappa_\dims}{\vol_\dims(K)}}\sum_{m=2}^\infty\binom{m+\frac1p-1}{m}\frac{(-1)^mr_0^{1-mp-n(m+\frac1p+1)}}{\dims-1}\calr_1^{-1}\left(\calc_1^p\rho_K^{\dims+p}-r_0^{\dims+p}\right)^m\\
            \label{calip_gammma_Linearization_HOTR}
            &+\frac1{\dims-1}\sum_{m=2}^{\dims-1}\binom{\dims-1}{m}\left(\rho_K-r_0\right)^{m}r_0^{\dims-1-m}\\
            \label{calip_gammma_Linearization_HOTC-1}
            &+\frac{C_K'}{C_B'}\sqrt[p]{\frac{\kappa_\dims}{\vol_\dims(K)}}\frac1{p(\dims-1)}\sum_{m=2}^{\dims+p} \binom{\dims+p}{m}r_0^{(\frac1p-m)(\dims+p)-1}(\calr^{-1}\circ\calc_1^p)(\rho_k-r_0)^m.
        \end{align}
         Assuming that $\norm[\infty]{\rho_K-r_0}<\epsilon$, we observe a few facts for the above equation.  The linear term \eqref{calip_gammma_Linearization_linear} consists of a constant of the order $1+O(\epsilon)$ multiplying an  operator that will be shown to be a contraction.  The higher order terms \eqref{calip_gammma_Linearization_HOTC}, \eqref{calip_gammma_Linearization_HOTR} and \eqref{calip_gammma_Linearization_HOTC-1} will be shown to be of the order of $O(\epsilon^2)$.  
\subsection{Proof of Theorem \ref{ci}} We will assume for now that the terms \eqref{calip_gammma_Linearization_HOTC}-\eqref{calip_gammma_Linearization_HOTC-1} are indeed of the order $O(\epsilon^2)$, and focus on terms \eqref{calip_gammma_Linearization_constant} and \eqref{calip_gammma_Linearization_linear}.  
   Taking $L^2$ norms,  we obtain,
    \begin{align*}
        &\norm[2]{\rho_K-r_0}\leq\norm[2]{\rho_K-r_0-\frac{C_K'}{C_B'}\sqrt[p]{\frac{\kappa_\dims}{\vol_\dims(K)}}\frac{r_0^{\frac{\dims}p}}{\dims-1}-\frac{r_0^{\dims-1}}{\dims-1}}\\
        &\leq \left(\frac{C_K'}{C_B'}\sqrt[p]{\frac{\kappa_\dims}{\vol_\dims(K)}}r_0^{\frac{1-p}p+n+p}\norm[L^2\rightarrow L^2]{\frac{\dims+p}{p(\dims-1)}\calr_1^{-1}\circ\calc_1^p}+O(\epsilon)\right)\norm[2]{\rho_K-r_0},
    \end{align*}
    where $\norm[L^2\rightarrow L^2]\cdot$ is the operator norm.
    The first inequality follows from the orthogonality of the spherical harmonic spaces and the definition of $r_0$.  If we are able to show that the term in parenthesis on the right hand side is strictly less than one for any given $p\geq1$ and $\epsilon$ small enough, then there will exist some $\epsilon_0$ such that $\epsilon<\epsilon_0$ implies $\norm[2]{\rho_K-r_0}=0$ or, equivalently, $\rho_K=r_0$. 
    
   We will check the operator norm $\norm[L^2\rightarrow L^2]{\frac{\dims+p}{p(\dims-1)}\calr_1^{-1}\circ\calc_1^p}$, as the terms multiplying it have already been shown to be $1+O(\epsilon)$ in subsection 6.2.  The eigenspaces of the spherical Radon transform are the spherical harmonic spaces and the eigenvalues of the normalized spherical Radon transform can be seen below %groemer Eq 3.4.16 pg 103
   \[
        \calr\calp[\stdindex]=\sigma_{\dims-1}(-1)^{\frac \stdindex2}\mu_\stdindex\calp[\stdindex]=\begin{cases}
            \sigma_{\dims-1}\calp[\stdindex]&: \stdindex=0\\
            \sigma_{\dims-1}(-1)^{\frac \stdindex2}\frac{(k-1)!!(\dims-3)!!}{(\dims+k-3)!!}\calp[\stdindex]&: 2\mid \stdindex\text{ and }k\neq 0\\
            0 &:  2\nmid \stdindex,
        \end{cases}
   \]
   see for example \cite[Eq 3.4.16 pg 103]{Gro}.
   Combining this with the eigenvalues for the $p$-cosine transform, we have that
    \[
%        \abs{\frac{\dims+p}{p(\dims-1)}\calr_1^{-1}\circ\calc_1^p\circ\calp[2\stdindex]}=\abs{\frac{\dims+p}{p(\dims-1)}        \frac{(\dims+2\stdindex-3)!!}{(2\stdindex-1)!!(\dims+1)!!}%%Inverse Radon Transform not normalized
%        \prod_{\ell=0}^{k-1}\abs{\frac{2\ell-p}{2\ell+\dims+p}} %%p-Cosine TF
%        \calp[2\stdindex]}\\
        \norm[L^2\rightarrow L^2]{\frac{\dims+p}{p(\dims-1)}\calr_1^{-1}\circ\calc_1^p\circ\calp[2\stdindex]}=\frac{\dims+p}{p(\dims-1)}\prod_{\ell=0}^{k-1}\abs{\frac{2\ell-p}{2\ell+1}\,\,\frac{2\ell+\dims-1}{2\ell+\dims+p}}.
        \]
       For $k=1$, the right hand side equals 1. As $k$ increases, each additional term in the product is strictly less than 1.  Therefore this operator norm is less than $1$.
    
    \subsection{Higher order terms}
    The fact that \eqref{calip_gammma_Linearization_HOTR} and \eqref{calip_gammma_Linearization_HOTC-1} are of the order $O(\epsilon^2)$ is immediate since $m \geq 2$. This leaves just \eqref{calip_gammma_Linearization_HOTC} to be estimated.
     
     We need to introduce a family of operators to help estimate the composition of the inverse Radon transform and the powers of the $p$-cosine transform.  Let $\mathbbm{1}_I:\bbr\rightarrow \{0,1\}$ be the indicator function of the interval $[0,1]$, and define $M_k:L^2\rightarrow L^2$ by
    \[M_k=\sum_{j=0}^{\infty}\mathbbm{1}_I\left(\frac j{2^k}\right)\calp[j],\] 
    and $\widetilde{M}_0:=M_0$ and $\widetilde{M}_{k}:=M_{k+1}-M_{k}$.  Observe that
    \[\norm[L^2]{M_k f}\leq \norm[L^2]{f},\hspace{40pt}\norm[L^\infty]{\calp_k f}\leq \left(\frac{\dim(\calh_{k}^{\dims})}{\abs{\calsn}}\right)^{1/2}\norm[L^2]{f},\]
    where the last estimates can be found in \cite[Eq. 2.48, pg. 27]{AH}.  For $f\in C(\calsn)$, $M_kf\rightarrow f$ in $L^2$ as well.  Thus,  for any $m\in\bbn$
    \begin{align*}
        (f)^m&=\sum_{k=0}^\infty (M_{k+1}f)^m-(M_kf)^m\\
            &=\sum_{k=0}^\infty \left(M_{k+1}f-M_kf\right)\left(\sum_{j=0}^{m-1}(M_{k+1}f)^j(M_{k}f)^{m-1-j}\right)\\
            &=\sum_{k=0}^\infty \left(\widetilde{M}_{k}f\right)\left(\sum_{j=0}^{m-1}(M_{k+1}f)^j(M_{k}f)^{m-1-j}\right),
    \end{align*}
    where the last equality is uses the linearity of $M_k$ and the definition of $\widetilde{M}_k$.  Breaking each of the summands into their dyadic harmonic projections yields
    \begin{align*}
            (f)^m&=\sum_{b=0}^\infty\sum_{k=0}^\infty\widetilde{M}_{b}\left[\left(\widetilde{M}_{k}f\right)\left(\sum_{j=0}^{m-1}(M_{k+1}f)^j(M_{k}f)^{m-1-j}\right)\right].
    \end{align*}
    While the product of two harmonic polynomials need not be harmonic, any polynomial when restricted to the sphere may be written as the sum of homogeneous, harmonic polynomials of lesser or equal degree (see \cite[Lemma 3.2.5, pg. 69]{Gro} or \cite[Theorem 2.18, pg 31]{AH}).
    Thus, if the degree of the polynomial in square brackets is smaller than $2^b$, then the operator $\widetilde{M}_b$ will give $0$ identically. For non-zero terms, the degree of the polynomial in brackets is
    \[2^b\leq 2^{k+1}+(m-1)2^{k+1}=m2^{k+1}\implies b\leq k+\log_2(2m).\]
    Now, taking the inverse Radon transform and $L^2$ norms on both sides, we see that the right hand side is bounded by
    \begin{align*}
        %\norm[L^2]{\calr^{-1}_1f^m}\leq
        & \sum_{b=0}^{\infty}\norm[L^2]{\calr^{-1}_1\widetilde{M}_{b}\sum_{k=b-\lfloor\log_2(2m)\rfloor}^{\infty}\left[\left(\widetilde{M}_{k}f\right)\left(\sum_{j=0}^{m-1}(M_{k+1}f)^j(M_{k}f)^{m-1-j}\right)\right]}\\
        &\leq \sum_{b=0}^{\infty}\mu_{2^{b+1}}^{-1}\norm[L^2]{\sum_{k=b-\lfloor\log_2(2m)\rfloor}^{\infty}\left[\left(\widetilde{M}_{k}f\right)\left(\sum_{j=0}^{m-1}(M_{k+1}f)^j(M_{k}f)^{m-1-j}\right)\right]}\\
        &\leq \sum_{b=0}^{\infty}\mu_{2^{b+1}}^{-1}\sum_{k=b-\lfloor\log_2(2m)\rfloor}^{\infty}\norm[L^2]{\widetilde{M}_{k}f}\left(\sum_{j=0}^{m-1}\norm[L^\infty]{M_{k+1}f}^j\norm[L^\infty]{M_{k}f}^{m-1-j}\right).
    \end{align*}
    
    Setting $f=\calc_1^p g$, where $g=\rho_K^{\dims+p}-r_0^{\dims+p}$,       
    we obtain the following  bound on the $L^\infty$ norm,
    \begin{align*}
        \norm[L^\infty]{M_k\calc_1^p g}&\leq\sum_{a=0}^{2^{k}}\norm[L^\infty]{\calp[a]\calc_1^p g}\leq\frac{1}{\sqrt{\abs{\calsn}}}\sum_{a=0}^{2^{k}}\abs{\frac{m_{p,a}}{m_{p,0}}}\sqrt{\dim(\calhnd[a])}\norm[L^2]{\calp[a]g}\\
        &\leq\frac{\sup_a\{\abs{m_{p,a}}\sqrt{\dim(\calhnd[a])}\}}{\abs{m_{p,0}}\sqrt{\abs{\calsn}}}\norm[L^2]{g}.
    \end{align*}
    Observe that $\abs{m_{p,a}}=O(a^{-(\dims+2p)/2})$ and $\dim(\calhnd[a])=O(a^\dims)$, so this supremum is indeed finite though dependent on the dimension.  Simplifying,
    \begin{align}\label{summ}
        \norm[L^2]{\calr^{-1}(\calc_1^p g)^m}\leq
    \end{align}
    \[\leq mC^{m-1}\norm[L^2]{g}^{m}\left(\sum_{b=0}^{\lfloor\log_2(m)\rfloor}\mu_{2^{b+1}}^{-1}+\sum_{b=\lfloor\log_2(2m)\rfloor}^{\infty}\mu_{2^{p+1}}^{-1}\abs{\frac{m_{p,2^{b-\lfloor\log_2(2m)\rfloor}}}{m_{p,0}}}\right).\]
   where $C=\max_k\norm[L^\infty\rightarrow L^2]{M_k\calc_1^p}$.
   Now we are in a position to estimate the higher order term \eqref{calip_gammma_Linearization_HOTC}.  Including the binomial coefficient and power of $r_0$ changing with $m$, we see

    \begin{align}
        &\binom{m+\frac1p-1}{m}\frac{r_0^{m(p-n)}}{\dims-1}\norm[L^2]{\calr^{-1}(\calc_1^p g)^m}\leq \frac{(m+1)\Gamma(\frac1p)}{(1-\epsilon)^{m\abs{n-p}}}mC^{m-1}\norm[L^2]{g}^{m}\cdot\\\nonumber &\cdot\left(\sum_{b=0}^{\lfloor\log_2(m)\rfloor}\mu_{2^{b+1}}^{-1}+\sum_{b=\lfloor\log_2(2m)\rfloor}^{\infty}\mu_{2^{p+1}}^{-1}\abs{\frac{m_{p,2^{b-\lfloor\log_2(2m)\rfloor}}}{m_{p,0}}}\right).
    \end{align}
    The right hand side is of the order $O(\epsilon^{m})$ times the sum in parentheses.  Note that the sums are constant for $m\in [2^a,2^{a+1})$ and at these  values of $m$ we get a right hand side of
        \[
        \epsilon^{2^a}\left(\sum_{b=0}^{a}\mu_{2^{b+1}}^{-1}+\sum_{b=0}^\infty\mu_{2^{b+a+2}}^{-1}\abs{\frac{m_{p,2^{b}}}{m_{p,0}}}\right).\]
    Summing over $m\in [2^a,2^{a+1})\cap \bbn$, we then obtain a right hand side of
    \[\left[\frac{1-\epsilon^{2^{a}+1}}{1-\epsilon}\right]\epsilon^{2^a}\left(\sum_{b=0}^{a}\mu_{2^{b+1}}^{-1}+\sum_{b=0}^\infty\mu_{2^{b+a+2}}^{-1}\abs{\frac{m_{p,2^{b}}}{m_{p,0}}}\right).\]
    The term in square brackets  is $1+O(\epsilon)$ and will be omitted from now on.  Now, summing over $a\in\bbn$ we get an upper bound for \eqref{calip_gammma_Linearization_HOTC},
    \[\eqref{calip_gammma_Linearization_HOTC}\leq \sum_{a=1}^{\infty}\epsilon^{2^a}\left(\sum_{b=0}^{a}\mu_{2^{b+1}}^{-1}+\sum_{b=0}^\infty\mu_{2^{b+a+2}}^{-1}\abs{\frac{m_{p,2^{b}}}{m_{p,0}}}\right).\]
    Since $\mu_k^{-1}=O(k^{\frac\dims2})$ and $m_{p,k}=O(k^{-\frac{\dims+2p}2})$, we get the estimate
    
    \begin{align*}
        \eqref{calip_gammma_Linearization_HOTC}&\leq\sum_{a=1}^\infty \epsilon^{2^{a}}\left(\sum_{b=0}^{a}2^{(b+1)\frac{\dims}{2}}+\sum_{b=0}^\infty 2^{(b+a+2)\frac{\dims}2}2^{-b\frac{\dims+2}2}\right).\end{align*}
    Evaluating the partial sums of the geometric series, we have
    \begin{align*}
        \eqref{calip_gammma_Linearization_HOTC}\leq\sum_{a=1}^\infty \epsilon^{2^{a}}\left(\frac{2^{(a+2)\frac{\dims}{2}}-1}{2^{\frac{\dims}{2}}-1}+2^{(a+2)\frac{\dims}{2}+1}\right)
%        &\leq\epsilon^2\left(2^{\frac{3\dims}2+2}+\sum_{a=2}^\infty \epsilon^{2^{a}-2-a}\left(2^{a\log_2\epsilon+(a+2)\frac{\dims}{2}+2}\right)\right)\\
        \leq \frac{\epsilon^2}{1-\epsilon} C_\dims
    \end{align*}
    Finally, picking epsilon small enough then yields the desired estimate.
    
    \bigskip
    
    {\it Acknowledgements:}  I would like to recognize my advisor Mar\'ia de los \'Angeles Alfonseca for her guidance in this work.  I would also like to thank Fedor Nazarov for his advice to use the $\calm_p$ operators for the final estimate.  I also thank Dmitry Ryabogin and Vladyslav Yaskin for several fruitful discussions.  
   
    \bibliographystyle{plain}
    \bibliography{Main_Clean-Rev1}
     \end{document}